\documentclass[11pt,reqno]{amsart}

\usepackage{amssymb,dirtytalk,comment,tikz-cd,tikz,color,cite,enumerate,setspace,longtable}
\usepackage[unicode=true]
 {hyperref}
\hypersetup{
colorlinks=true,
urlcolor=black,
citecolor=blue,
linkcolor=blue,
}
\usepackage{xltabular}
\usepackage{pdflscape}

\newtheorem{theorem}{Theorem}[section]
\newtheorem{lem}[theorem]{Lemma}
\newtheorem{prop}[theorem]{Proposition}
\newtheorem{cor}[theorem]{Corollary}

\theoremstyle{definition}
\newtheorem{definition}[theorem]{Definition}
\newtheorem{example}[theorem]{Example}

\newtheorem{assumption}[theorem]{Hypothesis}
\theoremstyle{remark}

\numberwithin{equation}{section}

\newcolumntype{C}[1]{>{\centering\arraybackslash}p{#1}}

\newcommand{\Soc}{\mathrm{Soc}}
\newcommand{\Pb}{\mathrm{Pb}}
\newcommand{\Fix}{\mathrm{Fix}}
\newcommand{\Ann}{\mathrm{Ann}}

\newcolumntype{L}[1]{>{\raggedright\arraybackslash}p{#1}}
\newcolumntype{C}[1]{>{\centering\arraybackslash}p{#1}}
\newcolumntype{R}[1]{>{\raggedleft\arraybackslash}p{#1}}

\allowdisplaybreaks
\begin{document}

\title[A probabilistic approach to the YBE and skew braces]{A probabilistic approach to the\\Yang--Baxter equation and skew braces}

\author{Maria Ferrara}
\address{Dipartimento di Ingegneria, Facoltà di Ingegneria e Informatica, Università Pegaso, Napoli, Italy}
\email{maria.ferrara1@unipegaso.it}

\author{Marco Trombetti}
\address{Dipartimento di Matematica e Applicazioni ``Renato Caccioppoli'', Università di Napoli Federico II, Complesso Universitario Monte S. Angelo, Via Cintia, Napoli, Italy}
\email{marco.trombetti@unina.it}

\author{Cindy (Sin Yi) Tsang}
\address{Department of Mathematics, Ochanomizu University, 2-1-1 Otsuka, Bunkyo-ku, Tokyo, Japan}
\email{tsang.sin.yi@ocha.ac.jp}
\urladdr{http://sites.google.com/site/cindysinyitsang/}

\subjclass[2020]{Primary 16T25\\\mbox{\hspace{5.835cm} }Secondary 20P05\, 81R50\, 20N99}

\keywords{Yang--Baxter equation; set-theoretic solution; skew brace; \\indecomposable solution; commuting probability; probabilistic invariant}

\begin{abstract}
We investigate finite non-degenerate set-theoretic solutions to the Yang--Baxter equation and skew braces using a probabilistic approach. We introduce four probabilities that measure how far a solution is from being a flip, but in different ways. Our main results state that for solutions arising from skew braces, these probabilities exhibit a rigid behaviour --- apart from a finite list of exceptional values (which we show to occur by means of explicit examples), they admit an upper bound that is slightly above $\frac{1}{2}$. In the skew brace setting, these probabilities measure how far the underlying skew brace is from being a trivial brace (in two different ways: one via the annihilator and the other via the socle), a trivial skew brace, and an almost trivial skew brace. We also introduce a probability that is related to the indecomposable components of a solution, and a probability that measures how close an arbitrary bijective non-degenerate map is to being a solution. In contrast to our main results, these two probabilities do not exhibit a discrete behaviour near $1$ --- they can be made arbitrarily close to $1$. 
\end{abstract}

\maketitle

\section{Introduction}

The aim of this paper is to study solutions to the Yang--Baxter equation from a probabilistic perspective.  Recall that a {\it set-theoretic solution} to the Yang--Baxter equation  is a pair $(Z,r)$, where $Z$ is a non-empty set and
\[r: (x,y)\in Z\times Z\longmapsto (\lambda_x(y),\rho_y(x))\in Z\times Z\]
is a bijective map satisfying the consistency equation
\[ r_{12}r_{23}r_{12} = r_{23}r_{12}r_{23}.\]
Here $r_{12} = r \times \mathrm{id}_Z$ and $r_{23}=\mathrm{id}_Z\times r$.  We say that $(Z,r)$ is {\it non-degenerate} if the component maps $\lambda_x,\,\rho_y$ are bijective for all $x,y\in Z$, and {\it involutive} if $r^2 = \mathrm{id}_{Z\times Z}$. In the rest of the paper, the term {\it solution} will always mean non-degenerate set-theoretic solution to the Yang--Baxter equation, but not necessarily involutive.

\smallskip

Given any non-empty set $Z$, we always have the so-called {\it flip} solution
\[ r_{Z,\mbox{\tiny flip}} : (x,y) \in Z\times Z\longmapsto (y,x)\in Z\times Z.\]
This is the most trivial type of solution, and solutions $(Z,r)$ that resemble the flip are easier to understand. For example, this happens when $r$ acts as the flip on many pairs of elements of $Z$. This encompasses certain relevant types of solutions, including some particular cases of those considered in \cite{ballester}, as well as the so-called {\it multipermutation solutions}. The latter are solutions that can be retracted into the trivial solution over a singleton after finitely many natural identification steps. As mentioned in \cite{multipermut}, computer experiments suggest that almost all solutions are multipermutation. 
It turns out that (also noted in \cite{multipermut}) a solution is multipermutation if and only if the so-called {\it structure skew brace} (an algebraic structure attached to the solution) satisfies a suitable nilpotency condition, which in turn implies that the solution contains many ‘‘involutive’’ pairs of elements fixing each other under the action of~$\lambda$. In view of this, we introduce the following two probabilities. 

\begin{definition}\label{defn:Pb1}
Let $(Z,r)$ be a finite solution. Define
\[\Pb_{\mbox{\tiny flip},1}(Z,r)  = \frac{1}{|Z|^2} | \{(x,y) \in Z^2\,:\, r(x,y) =(y,x),\, r(y,x) = (x,y)\}|,\]
and similarly
\[\Pb_{\mbox{\tiny flip},2}(Z,r)  = \frac{1}{|Z|^2}|\{(x,y)\in Z^2 \,:\, r^2(x,y) = (x,y),\, \lambda_x(y) = y\} |.\]
For both $i=1,2$, we have
\begin{align*}
(Z,r)\mbox{ is a flip solution}  \iff \Pb_{\mbox{\tiny flip},i}(Z,r)=1,
\end{align*}
so these probabilities measure how far $(Z,r)$ is from being a flip, but in two different ways: the first in a direct and pointwise manner, while the second in a way that exploits  the relation with multipermutation solutions.
\end{definition}

The probability of being a flip can also be localized to the first or second component as follows. 

\begin{definition}\label{defn:Pb2}
Let $(Z,r)$ be a finite solution. Define
\[\Pb_{\mbox{\tiny left}}(Z,r)  = \frac{1}{|Z|^2} | \{(x,y) \in Z^2\,:\, \lambda_x(y)=y\}|,\]
and similarly
\[\Pb_{\mbox{\tiny right}}(Z,r)  = \frac{1}{|Z|^2}|\{(x,y)\in Z^2 \,:\, \rho_y(x)=x\} |.\]
\end{definition}

Our main results focus on the four probabilities introduced above, in the context of finite solutions derived from skew braces, as follows. Recall that a {\it skew \textnormal(left\textnormal) brace} is a set $B = (B,+,\circ)$ endowed with two group structures $(B,+)$ and~$(B,\circ)$ such that the {\it skew \textnormal(left\textnormal) distributivity law}
\[ a\circ (b+c) = a\circ b - a + a\circ c\]
holds for all $a,b,c\in B$. It is easy to show that $(B,+)$ and $(B,\circ)$ share the same identity element $0$. It is also well-known (e.g. see \cite{Childs},\cite{GV}) that we can construct a solution by defining 
\[ r_B : (a,b)\in B\times B\longmapsto (-a+a\circ b,(-a+a\circ b)^{-1}\circ a\circ b)\in B\times B, \]
and that $r_B$ is involutive if and only if $B$ is a brace (i.e.~if the additive group $(B,+)$ is abelian). Following \cite{lastpaper}, we refer to the solutions arising from skew braces in this way as {\it sb-solutions}. For example, we have
\[ r_B : (a,b)\in B\times B\longmapsto (b,-b + a +b)\in B\times B\]
when $B$ is {\it trivial}, i.e. $a\circ b = a+b$ for all $a,b\in B$, and
\[ r_B : (a,b)\in B\times B\longmapsto (-a+b + a,a)\in B\times B\]
when $B$ is {\it almost trivial}, i.e. $a\circ b = b+a$ for all $a,b\in B$. We show that for finite sb-solutions, the  probabilities in Definitions \ref{defn:Pb1} and \ref{defn:Pb2} exhibit a discrete behaviour below the value $1$. 


\medskip

\noindent\textbf{Theorem A.}\quad {\it Let $\Pb$ be any of the probabilities $\Pb_{\mathrm{\tiny flip,1}}, \Pb_{\mathrm{\tiny flip,2}},\Pb_{\mathrm{\tiny left}}$, and $\Pb_{\mathrm{\tiny right}}$. Let $p$ be a prime. For any finite sb-solution $(Z,r)$, we have\textnormal:
\begin{enumerate}[$(a)$]
\item If $|Z|$ is odd and admits $p$ as its smallest prime divisor, then either
\[
\Pb(Z,r)\leq \frac{2p+1}{p(p+2)}
\quad \mbox{or}\quad
\Pb(Z,r)\in \left\{\frac{2p-1}{p^2},1\right\}.\]
\item If $|Z|$ is a power of $p$, then either
\[
\Pb(Z,r)\leq \frac{p^2+p-1}{p^3}
\quad \mbox{or}\quad
\Pb(Z,r)\in \left\{\frac{2p-1}{p^2},1\right\}.
\]
\end{enumerate}
}

\medskip

The odd-order assumption in Theorem A(a) excludes the case $p=2$; the even-order behaviour is treated in Theorems B and C, where a better bound than $\frac{5}{8} = \frac{2\cdot 2 + 1}{2(2+2)}$ is obtained and there are more discrete values.


\medskip

\noindent\textbf{Theorem B.}\quad {\it Let $\Pb$ be any of the probabilities $\Pb_{\mathrm{\tiny flip,2}},\Pb_{\mathrm{\tiny left}}$, and $\Pb_{\mathrm{\tiny right}}$. For any finite sb-solution $(Z,r)$, we have
\[ \Pb(Z,r) \leq \frac{7}{12}\quad \mbox{or}\quad \Pb(Z,r)\in \left\{\frac{3}{5},\frac{5}{8},\frac{2}{3},\frac{3}{4},1\right\}.\]
}

\medskip

We omitted the probability $\Pb_{\mathrm{\tiny flip,1}}$ in Theorem B because by exploiting the symmetric property of its defining conditions, we can get a better upper bound than $\frac{7}{12}$, and there are fewer possibilities for the discrete values.

\medskip

\noindent \textbf{Theorem C.}\quad {\it For any finite sb-solution $(Z,r)$, we have
\[ \Pb_{\mathrm{\tiny flip,1}}(Z,r) \leq \frac{9}{16}\quad \mbox{or}\quad \Pb_{\mathrm{\tiny flip,1}}(Z,r)\in \left\{\frac{5}{8}, \frac{3}{4},1\right\}.\]
}

\medskip

We also show that all of the discrete values appearing in Theorems A, B, and C indeed occur as the probabilities under consideration of some finite sb-solution (see the examples in Section \ref{sec:example}). Moreover, we provide a complete characterisation, in terms of certain quantities associated to the skew brace from which the sb-solution arises, of when these discrete values are realised (see the tables at the end of Section \ref{sec:proof}).

\medskip

The probabilities in Definitions \ref{defn:Pb1} and \ref{defn:Pb2} can also be related to natural invariants and structural notions of skew braces. In fact, let $B = (B,+,\circ)$ be a (finite) skew brace. For any $a,b\in B$, write
\[ a\ast b = -a + a\circ b - b,\quad [a,b]_+ = a + b - a - b.\]
Recall that the {\it annihilator} of $B$ is defined by
\[\Ann(B)= \{a\in B \mid\forall b\in B: a \ast b = [a,b]_+ = b\ast a = 0\},\]
and the {\it socle} of $B$ is defined by
\[ \Soc(B) = \{a\in B \mid \forall b\in B: a\ast b=[a,b]_+=0\}.\]
We will show in Propositions \ref{prop:solution1}(i) and \ref{prop:solution1}(ii), respectively, that
\[
\Pb_{\mathrm{\tiny flip,1}}(B,r_B)  = \frac{1}{|B|^2} |\{(a,b) \in B\times B \,:\, a \ast b = [a,b]_+= b\ast a=0\}|,
\]
which is the {\it commuting probability} considered in \cite{Manoj}, and similarly
\[
\Pb_{\mathrm{\tiny flip,2}}(B,r_B)  = \frac{1}{|B|^2} |\{(a,b) \in B\times B \,:\, a\ast b = [a,b]_+=0\}|.\]
Thus, from the skew brace perspective, these probabilities measure how far~$B$ is from being a trivial brace, but again in two different ways: the first via the annihilator while the second via the socle. It is clear that
\[
\Pb_{\mathrm{\tiny left}}(B,r_B)  = \frac{1}{|B|^2} |\{(a,b) \in B\times B \,:\,a \circ b = a+b \}|,
\]
and we will show in Proposition \ref{prop:solution1}(iii) that
\[
\Pb_{\mathrm{\tiny right}}(B,r_B)  = \frac{1}{|B|^2} |\{(c,a) \in B\times B \,:\, c\circ a=a+c  \}|.
\]
Thus, from the skew brace perspective, these probabilities measure how far~$B$ is from being trivial and almost trivial, respectively.

\smallskip

In the last two sections, we will introduce two more probabilities about arbitrary solutions. Section~\ref{sec:indecomp} concerns the probability that two elements of a solution lie in the same indecomposable component. We provide lower and upper bounds for it (see Theorem \ref{thm:component-bounds}). For
 example, if this probability is less than $1/k$, then the solution has at least $k+1$ indecomposable components (see Corollary~\ref{cor:number-components}).  Section~\ref{sec:almost-yang-baxter-maps} concerns the probability that a bijective non-degenerate map is a solution. It is worth emphasising that in contrast to our main results above, these two probabilities do not admit a uniform gap below~$1$ --- they can be made arbitrarily close to $1$ (see~The\-o\-rem~\ref{thm:no-uniform-yb-gap} and the remark at the end of Section \ref{sec:indecomp}).
 
\section{A general probability framework}\label{sec:general}

Before delving into the proof of Theorems A, B, and C, we introduce a general framework for probabilities attached to statements on a Cartesian product. This first part is purely set-theoretic and does not need any skew brace structure. Thus, let $X, \, Y$ be any non-empty finite sets, and let $W$ be a collection of statements $w(x,y)$ with variables $x\in X,\, y\in Y$. Below, we will adopt the convention that $\frac{1}{0}=\infty$ and $\frac{1}{\infty}=0$.


\begin{definition}\label{defn:W}
The \textit{{$W$-probability of the pair $(X,Y)$}} is
\[\Pb_W(X,Y) = \frac{1}{|X||Y|} |W(X,Y)|,
\] where \[ W(X,Y) = \{(x,y) \in X\times Y \mid\forall w\in W: w(x,y)\}.\]
This is the probability that a randomly selected element of $X\times Y$ satisfies all of the statements in $W$. Moreover, we define
\begin{align*}
    F_{1,W}(X,Y) &= \{x\in X \mid \forall y\in Y: (x,y) \in W(X,Y)\},\\
    F_{2,W}(X,Y) &= \{y\in Y \mid \forall x\in X: (x,y)\in W(X,Y)\}.
\end{align*}
For each $x\in X$, we further define
\[ C_W(x;Y) = \{y\in Y \mid (x,y) \in W(X,Y)\}.\]
Observe that by definition, we have
\begin{equation}\label{eqn:FC} F_{2,W}(X,Y) = \bigcap_{x\in X} C_W(x;Y) = Y\cap \Bigg(\bigcap_{x\in X\setminus F_{1,W}(X,Y)} C_W(x;Y)\Bigg) ,\end{equation}
which we interpret as $Y$ when the intersection is empty.
\end{definition}

It is clear from the definition that
\[
F_{1,W}(X,Y) = X \iff \Pb_W(X,Y) = 1 \iff F_{2,W}(X,Y) = Y.
\] 
This case is not interesting, so in what follows we assume that 
\begin{equation}\label{eqn:assumption} F_{1,W}(X,Y) \subsetneq X,\quad F_{2,W}(X,Y)\subsetneq Y.\end{equation}
Observe that we can write
\begin{align}\notag
    \Pb_W(X,Y) & = \frac{1}{|X||Y|}\sum_{x\in X}|C_W(x;Y)|\\\notag
    & = \frac{1}{|X||Y|} \Bigg( \sum_{x\in F_{1,W}(X,Y)} |Y| + \sum_{x\in X\setminus F_{1,W}(X,Y)} |C_W(x;Y)|\Bigg)\\\label{eqn:PbW}
    &=\frac{|F_{1,W}(X,Y)|}{|X|} + \frac{1}{|X|}\sum_{x\in X\setminus F_{1,W}(X,Y)}\frac{|C_W(x;Y)|}{|Y|}.
\end{align}
In order to estimate this probability, we put
\[ f_{1,W}(X,Y) = \frac{|X|}{|F_{1,W}(X,Y)|},\quad
f_{2,W} (X,Y)= \frac{|Y|}{|F_{2,W}(X,Y)|}. 
\]
For each $x\in X$, we also put
\[ c_W(x;Y) = \frac{|Y|}{|C_W(x;Y)|}.
\]
Only the elements $x\not\in F_{1,W}(X,Y)$ are of interest here, and we define
\[ c_W(X,Y) = \min\{c_W(x;Y) : x\in X\setminus F_{1,W}(X,Y)\}.\]
Note that this is $\infty$ if and only if $C_W(x;Y)$ is empty for all $x\not\in F_{1,W}(X,Y)$. It will also be helpful to consider the symmetric expression
\[ \varphi(u,v) = \frac{1}{u}+ \left(1-\frac{1}{u}\right)\frac{1}{v}= \frac{1}{v} + \left(1-\frac{1}{v}\right)\frac{1}{u},\mbox{ for }u,v \in [1,\infty].\]
Clearly $\varphi(u,v)$ is decreasing both as a function of $u$ and as a function of $v$. This simple fact is important when we execute the estimates. Since $\frac{1}{\infty}= 0$ by our convention, we have $\varphi(\infty,u)=\varphi(u,\infty)=1/u$, which is interpreted as $0$ when $u$ is also $\infty$.

\begin{prop}\label{prop:W} We always have
\[\Pb_W(X,Y) \leq \varphi(c_W(X,Y),f_{1,W}(X,Y)).\]
Moreover, we have the equality
\begin{align*}
\Pb_W(X,Y) &=\varphi(f_{2,W}(X,Y),f_{1,W}(X,Y)),\end{align*}
 provided that
\begin{align}\label{eqn:constant C}
C_W(x;Y) &= C_W(x';Y) \mbox{ for all }x,x'\in X\setminus F_{1,W}(X,Y).
\end{align}
\end{prop}
\begin{proof}
    It follows immediately from \eqref{eqn:PbW} that
\begin{align*}
\Pb_W(X,Y) & \leq \frac{|F_{1,W}(X,Y)|}{|X|} + \frac{1}{|X|} \cdot |X\setminus F_{1,W}(X,Y)|\cdot \frac{1}{c_W(X,Y)}\\
& = \frac{1}{|X|} \Bigg( |F_{1,W}(X,Y)| + \frac{1}{c_W(X,Y)}\cdot |X\setminus F_{1,W}(X,Y)|\Bigg)\\
&= \frac{1}{|X|} \Bigg( \frac{1}{c_W(X,Y)}\cdot |X| + \Bigg(1-\frac{1}{c_W(X,Y)}\Bigg) |F_{1,W}(X,Y)| \Bigg)\\
& = \frac{1}{c_W(X,Y)} + \Bigg(1-\frac{1}{c_W(X,Y)}\Bigg)\frac{1}{f_{1,W}(X,Y)},
\end{align*}
and this yields the first claim. Under the hypothesis \eqref{eqn:constant C}, the inequality in the above calculation is in fact an equality, and \eqref{eqn:FC} gives 
\[F_{2,W}(X,Y)=C_W(x;Y)\mbox{ for all }x\in X\setminus F_{1,W}(X,Y).\]
In this case, we have $c_W(X,Y)=f_{2,W}(X,Y)$, whence the second claim.
\end{proof}

We now specialise the preceding set-theoretic notation to the skew-brace setting. Let $B = (B,+,\circ)$ be a finite skew brace, and take $X=Y=B$ in the above notation. For simplicity, we write 
\[
\begin{aligned}
\Pb_W(B)&=\Pb_W(B,B),&
F_{i,W}(B)&=F_{i,W}(B,B) \,\ (i=1,2),\\
c_W(B)&=c_W(B,B),&f_{i,W}(B)&=f_{i,W}(B,B)\,\ (i=1,2).
\end{aligned}
\]


\begin{assumption}\label{assumption1} In what follows, we assume that for some $\bullet,\diamond\in \{+,\circ\}$, which need not be distinct, both of the conditions
\begin{enumerate}[$(1)$]
\item $F_{1,W}(B)$ is a subgroup of $(B,\bullet)$;
\item $C_W(a;B)$ is a subgroup of $(B,\diamond)$ for all $a\in B$;
\end{enumerate}
hold, and we further assume that
\begin{enumerate}[$(1)$]\setcounter{enumi}{+2}
\item $C_W(a;B)=C_W(a';B)$ for all $a,a'\in B\setminus F_{1,W}(B)$ if $f_{1,W}(B)$ is a prime.
\end{enumerate}
Moreover, we continue to assume \eqref{eqn:assumption}, so we have $\Pb_{W}(B) < 1$.
\end{assumption}

Conditions (1) and (2), together with \eqref{eqn:FC}, imply that
\[ f_{1,W}(B)\in \mathbb{N},\quad \forall a\in B: c_W(a;B)\in \mathbb{N},\quad c_{W}(B)\in \mathbb{N},
\quad f_{2,W}(B)\in \mathbb{N},\]
and they all divide $|B|$. Note that
\[ f_{1,W}(B),f_{2,W}(B) \geq 2,\quad  c_W(a;B) \geq c_W(B)\geq 2\mbox{ for }a\not\in F_{1,W}(B)\]
by \eqref{eqn:assumption} and by definition, respectively. Moreover, we have
\begin{equation}\label{eqn:inclusions} \forall a\in B:  F_{2,W}(B) \subseteq C_W(a;B) \subseteq B\end{equation}
by \eqref{eqn:FC}. These are subgroups of $(B,\diamond)$ by Condition (2), so then
\begin{equation}\label{eqn:divides}
\forall a\in B: c_{W}(a;B)\mbox{ divides } f_{2,W}(B). 
\end{equation}
We will need this observation for the last theorem in this section.

\smallskip

Here is our strategy. Let $f\geq 2$ be an integer, whose exact value will be specified later. For $f_{1,W}(B)\geq f$, it follows from Proposition \ref{prop:W} that
\begin{equation}\label{eqn:upper bound} \Pb_{W}(B) \leq \varphi(c_W(B),f_{1,W}(B))\leq \varphi(p,f),
\end{equation}
where $p$ is the smallest prime divisor of $|B|$. This gives us an upper bound, which is small when $f$ is big. Thus, we want to take $f$ as large as possible, but we also need to be able to understand what happens when
\[ 2\leq f_{1,W}(B)\leq f-1
\quad\mbox{and}\quad \Pb_W(B) > \varphi(p,f).\]
Condition (3) now comes in, because it allows us to exactly compute
\begin{equation}\label{eqn:prime}
    \Pb_W(B) = \varphi(f_{2,W}(B),f_{1,W}(B))=\varphi(f_{1,W}(B),f_{2,W}(B))
\end{equation}
using Proposition~\ref{prop:W} when $f_{1,W}(B)$ is a prime.

\begin{theorem}\label{thm:Pb1} Assume \textnormal{Hypothesis \textnormal{\ref{assumption1}}} and let $p$ be a prime.
\begin{enumerate}[$(a)$]
\item If $|B|$ admits $p$ as its smallest prime divisor, then either
\[ \Pb_W(B) \leq \varphi(p,p+2)\quad\mbox{or}\quad
\Pb_W(B) \in \{\varphi(2,3), \varphi(p,p)\}, 
\]
where $\varphi(2,3)$ occurs only when $p=2$, and we have
\begin{align*}
    \Pb_W(B) = \varphi(2,3) & \iff \{f_{1,W}(B),f_{2,W}(B)\} =\{2,3\},\\
    \Pb_W(B) = \varphi(p,p) & \iff f_{1,W}(B)=f_{2,W}(B)=p.
\end{align*}
\item If $|B|$ is a power of $p$, then either
\[ \Pb_W(B) \leq \varphi(p,p^2)\quad\mbox{or}\quad
\Pb_W(B) = \varphi(p,p),
\]
and we have
\[\Pb_W(B) = \varphi(p,p) \iff f_{1,W}(B)=f_{2,W}(B)=p.\]
\end{enumerate}
\end{theorem}
\begin{proof} Take $f=p+2$ in (a) and $f=p^2$ in (b), respectively. 
\begin{enumerate}[$\bullet$]
\item For $f_{1,W}(B) \geq f$, we have $\Pb_W(B)\leq \varphi(p,f)$ by \eqref{eqn:upper bound}. 
\item For $2\leq f_{1,W}(B)\leq f-1$, because $f_{1,W}(B)$ divides $|B|$, we see that
\[ f_{1,W}(B) = \begin{cases}
p & \mbox{in (a) when $p\geq 3$ and in (b)},\\
2,3 &\mbox{in (a) when $p=2$},
\end{cases}\]
where we have used the minimality of $p$ in (a). Since $f_{1,W}(B)$ is a prime in all cases, we may apply \eqref{eqn:prime} to deduce that
\[ \Pb_W(B) = \varphi(f_{2,W}(B),f_{1,W}(B))=\varphi(f_{1,W}(B),f_{2,W}(B)).\]
Now, either $f_{2,W}(B) \geq f$, or by the same reason as above
\[ f_{2,W}(B) = \begin{cases}
p & \mbox{in (a) when $p\geq 3$ and in (b)},\\
2,3 &\mbox{in (a) when $p=2$}.
\end{cases}\]
It is then easy to check that $\Pb_W(B) \leq \varphi(p,f) = \varphi(f,p)$ unless
\[ f_{1,W}(B) = f_{2,W}(B) = p,\]
and in (a) when $p=2$, we also have the possibilities
\[ (f_{1,W}(B),f_{2,W}(B)) = (2,3),(3,2).\]
These cases give us the exceptional values.
\end{enumerate}
This proves  all of the claims.
\end{proof}

In the case $p=2$, by allowing more discrete values, we can improve the bound $\varphi(2,4) = \frac{5}{8}$ in Theorem \ref{thm:Pb1}(a) by taking $f$ to be slightly bigger than~$4$. But in order to do this, we need the following additional hypothesis.

\begin{assumption}\label{assumption2}In what follows, assume further that:
\begin{enumerate}[(1)]\setcounter{enumi}{+3}
\item $|C_W(a;B)|= |C_W(a\bullet z;B)|$ for all $a\in B$ and $z\in F_{1,W}(B)$.
\end{enumerate}
Here $\bullet$ is the same operation in Hypothesis \ref{assumption1}.
\end{assumption}

By Condition (1), we may write
\[ B = F_{1,W}(B) \sqcup \bigsqcup_{i=1}^{f_{1,W}(B)-1} a_i\bullet F_{1,W}(B).\]
From Condition (4) and \eqref{eqn:PbW}, we then see that
\begin{align}\notag
\Pb_W(B) & = \frac{|F_{1,W}(B)|}{|B|}+\frac{1}{|B|}\sum_{i=1}^{f_{1,W}(B)-1} \sum_{a\in a_i\bullet F_{1,W}(B)}\frac{|C_{W}(a;B)|}{|B|}\\\notag
&=\frac{|F_{1,W}(B)|}{|B|}+\frac{|F_{1,W}(B)|}{|B|}\sum_{i=1}^{f_{1,W}(B)-1}\frac{|C_{W}(a_i;B)|}{|B|}\\\label{eqn:Pb3}
&= \frac{1}{f_{1,W}(B)} + \frac{1}{f_{1,W}(B)}\sum_{i=1}^{f_{1,W}(B)-1}\frac{1}{c_{W}(a_i;B)}.
\end{align}
Thus, even when $f_{1,W}(B)$ is not a prime, we can still compute $\Pb_W(B)$ as long as we can control the values of $c_{W}(a_i;B)$ for $i=1,\dots,f_{1,W}(B)-1$. It will be helpful to recall that they are integers greater than or equal to $2$.

\begin{theorem}\label{thm:Pb2} 
Assume \textnormal{Hypotheses} \textnormal{\ref{assumption1}} and \textnormal{\ref{assumption2}}. Then, either
\[ \Pb_W(B) \leq \varphi(2,6)\quad\mbox{or}\quad\Pb_W(B) \in \{\varphi(2,5),\varphi(2,4),\varphi(2,3),\varphi(2,2)\}, \]
and we have
\begin{align*}
\Pb_W(B) = \varphi(2,5) & \iff \{f_{1,W}(B),f_{2,W}(B)\} = \{2,5\},\\
\Pb_W(B) = \varphi(2,3) & \iff \{f_{1,W}(B),f_{2,W}(B)\} = \{2,3\},\\
\Pb_W(B) =  \varphi(2,2)& \iff f_{1,W}(B) = f_{2,W}(B) = 2,
\end{align*}
and
\begin{align*}
&
\Pb_W(B) = \varphi(2,4) 
\\ & \hspace{1em}\iff \begin{cases} f_{1,W}(B)= 2,\, f_{2,W}(B)=4, \, \mbox{or}\\
f_{1,W}(B)=4,\, \forall a\in B\setminus F_{1,W}(B): c_{W}(a;B)=2.
\end{cases}
\end{align*}
\end{theorem}
\begin{proof} Take $f = 6$ here.
\begin{enumerate}[$\bullet$]
\item For $f_{1,W}(B) \geq 6$, we have $\Pb_W(B) \leq \varphi(2,6)$ by \eqref{eqn:upper bound}.
\item For $f_{1,W}(B) =2,3,5$, this is a prime and so \eqref{eqn:prime} yields
\[ \Pb_W(B) = \varphi(f_{2,W}(B),f_{1,W}(B)) =\varphi(f_{1,W}(B),f_{2,W}(B)).\]
It is easy to check that $\Pb_W(B) \leq \varphi(2,6) = \varphi(6,2)$ unless
\[ (f_{1,W}(B),f_{2,W}(B)) \in  \{(2,2),(2,3),(2,4),(2,5),(3,2),(5,2)\},\]
where we have used the fact that $f_{2,W}(B)\geq 2$ is an integer.
\item For $f_{1,W}(B) =4$, we apply \eqref{eqn:Pb3} to deduce that 
\[ \Pb_W(B) =\frac{1}{4} + \frac{1}{4}\left(\frac{1}{2}+\frac{1}{2}+\frac{1}{2}\right) =\varphi(2,4)\]
if $c_W(a_i;B) = 2$ for every $i$, and
\[ \Pb_W(B) \leq \frac{1}{4} + \frac{1}{4}\left(\frac{1}{2}+\frac{1}{2}+\frac{1}{3}\right) =\varphi(2,6)\]
otherwise.
\end{enumerate}
This proves  all of the claims.
\end{proof}

We can improve the bound $\varphi(2,6) = \frac{7}{12}$ in Theorem \ref{thm:Pb2} even more, but we further need the following assumption.

\begin{assumption}\label{assumption3} In what follows, assume in addition that:
\begin{enumerate}[(1)]\setcounter{enumi}{+4}
\item $f_{1,W}(B) = f_{2,W}(B)$.
\end{enumerate}
Note that this is always the case when $W$ is {\it symmetric on $B$}, that is
\[ \forall a,b\in B: (a,b) \in W(B,B) \iff (b,a) \in W(B,B).\]
In fact, we even have $F_{1,W}(B) = F_{2,W}(B)$ when $W$ is symmetric on $B$.
\end{assumption}

Condition (5) not only allows us to replace $f_{2,W}(B)$ by $f_{1,W}(B)$ in \eqref{eqn:prime}, but also imposes restrictions on the $c_{W}(a_i;B)$ in \eqref{eqn:Pb3} by \eqref{eqn:divides}. This  significantly reduces the possibilities. 

\begin{theorem}\label{thm:Pb3} Assume \textnormal{Hypotheses \textnormal{\ref{assumption1}}}, \textnormal{\ref{assumption2}}, and \textnormal{\ref{assumption3}}. Then, either
\[ \Pb_W(B) \leq \varphi(2,8)\quad\mbox{or}\quad
\Pb_W(B) \in \{\varphi(2,6),\varphi(2,4),\varphi(2,2)\},\]
where $\varphi(2,6)$ does not occur when $W$ is symmetric on $B$ and $F_{1,W}(B) = F_{2,W}(B)$ is normal in $(B,\diamond)$, and we have
\begin{align*}
    \Pb_W(B) = \varphi(2,6) &\iff f_{1,W}(B)=6,\, \forall a\in B\setminus F_{1,W}(B): c_W(a;B)=2,\\
    \Pb_W(B) = \varphi(2,4) &\iff f_{1,W}(B)=4,\, \forall a\in B\setminus F_{1,W}(B): c_W(a;B)=2,\\
    \Pb_W(B) = \varphi(2,2) & \iff f_{1,W}(B) = 2.
\end{align*}
\end{theorem}
\begin{proof} Take $f=8$ here.
\begin{enumerate}[$\bullet$]
    \item For $f_{1,W}(B) \geq 8$, we have $\Pb_W(B) \leq \varphi(2,8)$ by \eqref{eqn:upper bound}.
    \item For $f_{1,W}(B) = 2,3,5,7$, this is a prime and so \eqref{eqn:prime} yields
    \[ \Pb_W(B) = \varphi(f_{2,W}(B),f_{1,W}(B))=\varphi(f_{1,W}(B),f_{1,W}(B)).\]
    It is easy to check that $ \Pb_W(B) \leq \varphi(2,8)$ unless $f_{1,W}(B)=2$.
\item For $f_{1,W}(B)=4$, notice that $c_{W}(a;B) = 2,4$ for all $a\in B\setminus F_{1,W}(B)$ by \eqref{eqn:divides} and Condition (5). We then apply \eqref{eqn:Pb3} to deduce that
\[ \Pb_W(B) =\frac{1}{4} + \frac{1}{4}\left(\frac{1}{2}+\frac{1}{2}+\frac{1}{2}\right) =\varphi(2,4)\]
if $c_W(a_i;B)=2$ for every $i$, and 
\[ \Pb_W(B) \leq \frac{1}{4} + \frac{1}{4}\left(\frac{1}{2}+\frac{1}{2}+\frac{1}{4}\right) =\varphi(2,8)\]
otherwise.
\item For $f_{1,W}(B) = 6$, we apply \eqref{eqn:Pb3} to deduce that
\[ \Pb_W(B) = \frac{1}{6} +\frac{1}{6}\left(\frac{1}{2}+\frac{1}{2}+\frac{1}{2}+\frac{1}{2}+\frac{1}{2}\right) = \varphi(2,6)\]
if $c_W(a_i;B)=2$ for every $i$, and
\[ \Pb_W(B) \leq  \frac{1}{6} +\frac{1}{6}\left(\frac{1}{2}+\frac{1}{2}+\frac{1}{2}+\frac{1}{2}+\frac{1}{3}\right) = \varphi(2,9) < \varphi(2,8)\]
otherwise.
\end{enumerate}
This proves all of the claims, except the one about $\varphi(2,6)$ not occurring as a possibility under certain assumptions.

\smallskip

Finally, suppose that $\Pb_W(B) = \varphi(2,6)$. We already know that then
\[ f_{1,W}(B) = 6,\quad \forall a\in B\setminus F_{1,W}(B) : c_W(a;B) = 2.\]
With respect to the operation $\diamond$, if $F_{2,W}(B)$ is normal in $B$, then let 
\[ Q = (B,\diamond)/(F_{2,W}(B),\diamond)\]
denote its quotient group. Since $|Q|= f_{2,W}(B)=f_{1,W}(B) =6$, it contains a unique subgroup $C/F_{2,W}(B)$ of order $3$. From \eqref{eqn:inclusions}, we then see that
\[ \forall a\in B\setminus F_{1,W}(B):C=C_W(a;B),\]
because $c_{W}(a;B)=2$ here. If $W$ is symmetric on $B$ in addition, then 
\[ \forall a,b\in B: a\in C_W(b;B) \iff b\in C_W(a;B).\]
We also have $F_{1,W}(B) = F_{2,W}(B)$, and hence
\[ \forall a\in B\setminus C,\, b\in C\setminus F_{2,W}(B): a\not\in C = C_W(b;B),\,  b \in C = C_W(a;B),\]
which is a contradiction.
\end{proof}

\section{Proof of Theorems A, B, and C}\label{sec:proof}

We apply the skew-brace specialisation of Section \ref{sec:general} to the collections 
\begin{align*}
W_{\operatorname{ann}}&=\{x\ast y=0,\, [x,y]_+=0,\, y\ast x=0\},\\
W_{\operatorname{soc}}&=\{x\ast y=0,\, [x,y]_+=0\}, \\
W_{\operatorname{tri}}&=\{\lambda_x(y)=y\},\\
W_{\operatorname{atri}}&=\{\lambda^{\operatorname{op}}_x(y)=y\},
\end{align*}
of statements with variables $x,y$ in any (finite) skew brace $B = (B,+,\circ)$. We will often use the well-known fact that
\begin{align*}
\lambda : a\in (B,\circ) &\longmapsto \lambda_a\in \mathrm{Aut}(B,+),&&\hspace{-0.75cm}(\lambda_a(b) = -a + a\circ b),\\
\lambda^{\operatorname{op}}: a\in (B,\circ) &\longmapsto \lambda_a^{\operatorname{op}}\in \mathrm{Aut}(B,+),&&\hspace{-0.75cm}(\lambda_a^{\operatorname{op}}(b) =a\circ b  -a ),
\end{align*}
are group homomorphisms without reference. We observe that
\begin{equation}\label{eqn:prob equality}
\begin{cases}
\begin{aligned}
\Pb_{\mathrm{flip},1}(B,r_B) &= \Pb_{W_{\mathrm{ann}}}(B),&\Pb_{\mathrm{left}}(B,r_B) &= \Pb_{W_{\mathrm{tri}}}(B),\\
\Pb_{\mathrm{flip},2}(B,r_B) &= \Pb_{W_{\mathrm{soc}}}(B),&\Pb_{\mathrm{right}}(B,r_B) &= \Pb_{W_{\mathrm{atri}}}(B) .
\end{aligned}\end{cases}
 \end{equation}
The equality about the probabilities $\Pb_{\mathrm{left}},\, \Pb_{W_{\mathrm{tri}}}$ is trivial, and the others follow from Proposition \ref{prop:solution1} below. Here, recall that
\[ r_B: (a,b)\in B\times B \longmapsto (\lambda_a(b),\lambda_a(b)^{-1}\circ a\circ b)\in B\times B,\]
and $\rho_b(a)$ denotes the second coordinate in the image.

\begin{prop}\label{prop:solution1} Let $B=(B,+,\circ)$ be any skew brace.
\begin{enumerate}[$(i)$]
\item For any $a,b\in B$, we have
\[ 
r_B(a,b)=(b,a),\, r_B(b,a) = (a,b)
 \iff a*b =[a,b]_+ = b*a = 0.\]
 \item  
 For any $a,b\in B$, we have
\[ r_B^2(a,b) = (a,b),\, \lambda_a(b) =b  \iff 
a*b=[a,b]_+ = 0.\]
\item For $a,b\in B$ and $c = a\circ b\circ a^{-1}$, we have
\[ \rho_b(a) = a \iff c \circ a=a+c \iff \lambda^{\mathrm{\tiny op}}_{c}(a)=a.\]
\end{enumerate}
\end{prop}
\begin{proof} For (i), observe that
\begin{align*} 
&r_B(a,b) = (b,a),\, r_B(b,a) = (a,b)\\
 &\hspace{1em} \iff \lambda_a(b) = b,\, \lambda_a(b)^{-1}\circ a\circ b = a,\, \lambda_b(a)=a,\, \lambda_b(a)^{-1}\circ b\circ a = b\\
&\hspace{1em} \iff a \ast b = 0,\, a\circ b =b\circ a,\, b\ast a=0\\
&\hspace{1em} \iff a \ast b = 0,\, [a,b]_+ = 0,\, b\ast a=0,
\end{align*}
where the last equivalence follows because $a\ast b=0$ and $b\ast a=0$ simply mean that $a\circ b = a+b$ and $b\circ a = b+a$, respectively.

\smallskip

For (ii), we may assume that $\lambda_a(b)=b$, which is equivalent to $a\ast b=0$. On the one hand, since $\lambda_a(b) = b$ and $a\circ b = a+b$, we see that
\[
\rho_b(a) = b^{-1} \circ (a+b) = (b^{-1}\circ a ) - b^{-1} = b^{-1} + \lambda_{b^{-1}}(a) - b^{-1}.\]
But $\lambda_b(b^{-1})=-b$, so this implies that 
\((\lambda_b\rho_b)(a) = -b + a +b.\)
On the other hand, we have $\rho_b(a) = b^{-1}\circ (a\circ b)$ as above, and so
\[ \rho_{\rho_b(a)}(b)  = (\lambda_b\rho_b)(a)^{-1}\circ b \circ \rho_b(a)=(\lambda_b\rho_b)(a)^{-1}\circ a\circ b,\]
which is $b$ when $(\lambda_b\rho_b)(a)=a$. Thus, assuming that $\lambda_a(b)=b$, we have
\begin{align*} 
r_B^2(a,b) = (a,b) 
&\iff (\lambda_{\lambda_a(b)}\rho_b)(a) = a,\, (\rho_{\rho_b(a)}\lambda_a)(b) = b\\
& \iff (\lambda_b\rho_b)(a) = a ,\, \rho_{\rho_b(a)}(b)=b\\
& \iff [a,b]_+ =0,
\end{align*}
where the last equivalence holds by the above observations.

\smallskip

For (iii), note that $a\circ b = c\circ a$, so then
\begin{align*}
\rho_b(a) = a & \iff \lambda_a(b)^{-1} \circ a\circ b = a\\
& \iff c \circ a= (-a + c \circ a)\circ a\\
& \iff a + c = c\circ a\\
& \iff \lambda^{\mathrm{\tiny op}}_{c}(a)=a,
\end{align*}
which is as desired.
\end{proof}

For any skew brace $B = (B,+,\circ)$, it is clear that
\begin{equation}\label{eqn:F}
\begin{cases}
\begin{aligned}
F_{1,W_{\operatorname{ann}}}(B)&= \Ann(B),&F_{1,W_{\operatorname{tri}}}(B)  &= \operatorname{Ker}(\lambda),\\
F_{1,W_{\operatorname{soc}}}(B)&= \Soc(B),&F_{1,W_{\operatorname{atri}}}(B)  &= \operatorname{Ker}(\lambda^{\operatorname{op}}),
    \end{aligned}
            \end{cases}
   \end{equation}
and for any $a\in B$, we have
\begin{equation}\label{eqn:C}
\begin{cases}
\begin{aligned}
    C_{W_{\operatorname{ann}}}(a;B) &= \Fix(\lambda_a) \cap C_{(B,+)}(a) \cap C_{(B,\circ)}(a),\\
C_{W_{\operatorname{soc}}}(a;B)  &= \Fix(\lambda_a)\cap C_{(B,+)}(a) ,\\
C_{W_{\operatorname{tri}}}(a;B) &= \Fix(\lambda_a),\\
C_{W_{\operatorname{atri}}}(a;B) &=\Fix(\lambda_a^{\operatorname{op}}).
\end{aligned}
\end{cases}
\end{equation}
Here $C_{G}(g)$ denotes the centraliser of $g$ in a group $G$, and $\mathrm{Fix}(\psi)$ denotes the set of fixed points for any permutation $\psi$.

\begin{theorem}\label{thm1} Let $W$ be any of the $W_{\operatorname{ann}},W_{\operatorname{soc}},W_{\operatorname{tri}}$, and $W_{\operatorname{atri}}$. Then $W$ satisfies \textnormal{Hypotheses \textnormal{\ref{assumption1}}} and \textnormal{\ref{assumption2}} for any finite skew brace $B=(B,+,\circ)$.
\end{theorem}
\begin{proof} For Condition (1), it is clear from \eqref{eqn:F} that $F_{1,W}(B)$ is a subgroup, in fact a normal subgroup, of $(B,\circ)$. Thus, we take $\bullet$ to be $\circ$ in all cases.

\smallskip

 For Condition (2), it is known from \cite[Theorem 3.3]{colazzo} that $C_{W_{\mathrm{ann}}}(a;B)$ is a subgroup of $(B,\circ)$, and it is clear from \eqref{eqn:C} that $C_W(a;B)$ is subgroup of~$(B,+)$ when $W = W_{\mathrm{soc}},W_{\mathrm{tri}},W_{\mathrm{atri}}$, for all $a\in B$.

\smallskip

For Condition (4), it is also clear from \eqref{eqn:F} and \eqref{eqn:C}. In fact, we have a set equality $C_W(a;B) = C_W(a\circ z ;B)$ for all $a\in B$ and $z\in F_{1,W}(B)$.

\smallskip

For Condition (3), assume that $f_{1,W}(B)$ is a prime, $q$ say. Then
\[ B = F_{1,W}(B)\sqcup \bigsqcup _{i=1}^{q-1} a^i\circ F_{1,W}(B)\]
for some $a\in B$ because $F_{1,W}(B)$ is a normal subgroup of $(B,\circ)$. In view of the set equality in ~Condition (4), it suffices to show that
\begin{equation}\label{eqn:CW}
C_W(a;B) = C_W(a^i;B)
\end{equation}
for all $1\leq i\leq q-1$. For $W=W_{\operatorname{ann}}$, simply observe that
\[ F_{1,W_{\operatorname{ann}}}(B) \subseteq C_{W_{\operatorname{ann}}}(a^i;B)\subsetneq B ,\]
and the primality of $f_{1,W_{\operatorname{ann}}}(B)$ yields 
\[ C_{W_{\operatorname{ann}}}(a^i;B) = F_{1,W_{\operatorname{ann}}}(B),\]
so in particular \eqref{eqn:CW} holds. For $W=W_{\operatorname{tri}},W_{\operatorname{atri}}$, we clearly have
\begin{align*}
    C_{W_{\operatorname{tri}}}(a;B) & = \Fix(\lambda_a) = \Fix(\lambda_a^i) = C_{W_{\operatorname{tri}}}(a^i;B),\\
    C_{W_{\operatorname{atri}}}(a;B) & = \Fix(\lambda^{\operatorname{op}}_a) = \Fix((\lambda^{\operatorname{op}}_a)^i) = C_{W_{\operatorname{atri}}}(a^i;B).
\end{align*}
For $W=W_{\operatorname{soc}}$, first note that
\[ F_{1,W_{\operatorname{soc}}}(B) \subseteq Z(B,+) \subseteq B,\]
and the primality of $f_{1,W_{\operatorname{soc}}}(B)$ forces $(B,+)/Z(B,+)$ to be cyclic. But this implies that $(B,+)$ is abelian, and we see that
\[C_{W_{\operatorname{soc}}}(a;B)  = \Fix(\lambda_a) =  \Fix(\lambda_a^i) = C_{W_{\operatorname{soc}}}(a^i;B)\]
as above. This completes the proof.
\end{proof}


We may now apply the theorems in Section \ref{sec:general} to deduce  Theorems A, B, and C. Note that in the latter, we have written out the explicit values
\begin{align*}
\varphi(p,v)&= \frac{2p+1}{p(p+2)} ,\frac{p^2 + p-1}{p^3}, \frac{2p-1}{p^2} &&\hspace{-0.25cm}\mbox{for } v=p+2,p^2,p,\\
\varphi(2,v) & =  \frac{9}{16},\frac{7}{12},\frac{3}{5},\frac{5}{8},\frac{2}{3},\frac{3}{4}&&\hspace{-0.25cm}\mbox{for }v=8,6,5,4,3,2.
\end{align*}
We have also added the value $1$ because we assumed \eqref{eqn:assumption} in Section \ref{sec:general}. 

\begin{proof}[Proof of Theorem A] This follows from \eqref{eqn:prob equality}, and Theorems \ref{thm:Pb1} and \ref{thm1}.
\end{proof}

\begin{proof}[Proof of Theorem B] This follows from \eqref{eqn:prob equality}, and Theorems \ref{thm:Pb2} and \ref{thm1}.
\end{proof}

\begin{proof}[Proof of Theorem C] Clearly $W_{\mathrm{ann}}$ is symmetric on $B$, and $F_{1,W_{\mathrm{ann}}}(B)$ is a normal subgroup of both $(B,+)$ and $(B,\circ)$ for any skew brace $B = (B,+,\circ)$. The theorem now follows from \eqref{eqn:prob equality}, and Theorems \ref{thm:Pb3} and \ref{thm1}.
\end{proof}

In Theorems A, B, and C, we only listed the spectra of the probabilities for
simplicity. But from Theorems~\ref{thm:Pb1}, \ref{thm:Pb2}, and
\ref{thm:Pb3}, we know exactly when the discrete values occur, which we
summarise in the following tables. Here $B=(B,+,\circ)$ is the finite skew
brace such that $(Z,r)=(B,r_B)$, and $W$ is the collection associated to the
probability under consideration as specified by \eqref{eqn:prob equality}.

\begin{longtable}{|C{2cm}|C{9cm}|}
\hline
 Value & Condition\\
\hline\hline
 $\frac{2p-1}{p^2}$ & $f_{1,W}(B)=f_{2,W}(B)=p$ \\\hline
\end{longtable}
\vspace{-3mm}
{\footnotesize \begin{center}
    {\textsc Table A.} Conditions for the discrete value to occur in Theorem A
\end{center}}

\smallskip

\begin{longtable}{|C{2cm}|C{9cm}|}
\hline
 Value & Condition\\
\hline\hline
 $\frac{3}{5}$ & $\{f_{1,W}(B),f_{2,W}(B)\}=\{2,5\}$ \\\hline
  $\frac{5}{8}$ & $\begin{aligned}&f_{1,W}(B)=2,\, f_{2,W}(B)= 4,\mbox{ or }\\[-5pt] &f_{1,W}(B)=4,\, \forall a\in B\setminus F_{1,W}(B):c_W(a;B)=2\end{aligned}$ \\\hline
   $\frac{2}{3}$ & $\{f_{1,W}(B),f_{2,W}(B)\}=\{2,3\}$ \\\hline
    $\frac{3}{4}$ & $f_{1,W}(B)=f_{2,W}(B)=2$ \\\hline
\end{longtable}
\vspace{-3mm}{\footnotesize \begin{center}
    {\textsc Table B.} Conditions for the discrete values to occur in Theorem B
\end{center}}

\smallskip

\begin{longtable}{|C{2cm}|C{9cm}|}
\hline
 Value & Condition\\
\hline\hline
  $\frac{5}{8}$ &$f_{1,W}(B)=4,\, \forall a\in B\setminus F_{1,W}(B):c_W(a;B)=2$ \\\hline
    $\frac{3}{4}$ & $f_{1,W}(B)=2$ \\\hline
\end{longtable}
\vspace{-3mm}{\footnotesize \begin{center}
    {\textsc Table C.} Conditions for the discrete values to occur in Theorem C
\end{center}}

\section{Examples realising the discrete values}\label{sec:example}

We now provide explicit examples to show that all of the discrete values in Theorems A, B, and C are realised by some finite sb-solution. Clearly $1$ occurs for all four types of probabilities. For the other values, see:
\begin{longtable}{|C{5.5cm}|C{1.75cm}|C{1.5cm}|c|}
\hline
Probability type& Value & Example & Remark\\
\hline\hline
$\Pb_{\mathrm{flip},1},\Pb_{\mathrm{flip},2},\Pb_{\mathrm{left}},\Pb_{\mathrm{right}}$ & $\frac{2p-1}{p^2},\, \frac{3}{4}$ &\ref{ex:p2} &\\\hline
$\Pb_{\mathrm{flip},2},\Pb_{\mathrm{left}},\Pb_{\mathrm{right}}$ & $\frac{3}{5}$ & \ref{ex:pq} & $(p,q)=(5,2)$ \\\hline
$\Pb_{\mathrm{flip},1},\Pb_{\mathrm{flip},2},\Pb_{\mathrm{left}},\Pb_{\mathrm{right}}$& $\frac{5}{8}$ & \ref{ex:p3} & $p=2$ \\\hline
$\Pb_{\mathrm{flip},2},\Pb_{\mathrm{left}},\Pb_{\mathrm{right}}$ & $\frac{2}{3}$ & \ref{ex:pq} & $(p,q)=(3,2)$ \\\hline
\end{longtable}
\noindent Let $B = (B,+,\circ)$ be a finite skew brace. In view of \eqref{eqn:prob equality}, we may compute the four probabilities of $(B,r_B)$ of interest as $\Pb_W(B)$, where $W$ is among the collections $W_{\mathrm{ann}},W_{\mathrm{soc}},W_{\mathrm{tri}},W_{\mathrm{atri}}$ as given by \eqref{eqn:prob equality}. By Theorem \ref{thm1}, in the computation, we may apply \eqref{eqn:prime} provided that $f_{1,W}(B)$ is a prime, and we may freely use \eqref{eqn:Pb3}. In the proof of Theorem \ref{thm1}, the operation $\bullet$ in Hypothesis \ref{assumption1} is taken to be $\circ$, so $0,a_1,\dots,a_{f_{1,W}(B)-1}$ in \eqref{eqn:Pb3} is a left transversal of $F_{1,W}(B)$ in~$(B,\circ)$. The following observations are obvious:
\begin{enumerate}[(I)]
\item If $(B,+)$ is abelian, then all of the sets in Definition \ref{defn:W} are the same for $W=W_{\mathrm{soc}},W_{\mathrm{tri}},W_{\mathrm{atri}}$, so in particular
\[\Pb_{W_{\mathrm{soc}}}(B) = \Pb_{W_{\mathrm{tri}}}(B) = \Pb_{W_{\mathrm{atri}}}(B).  \]
\item If $(B,\circ)$ is abelian, then all of the sets in Definition \ref{defn:W} are the same for $W=W_{\mathrm{ann}},W_{\mathrm{soc}}$, and we have
\[
\forall a,b\in B: (a,b) \in W_{\mathrm{tri}}(B,B) \iff (b,a)\in W_{\mathrm{atri}}(B,B),\]
so in particular
\[ \Pb_{W_{\mathrm{ann}}}(B)=\Pb_{W_{\mathrm{soc}}}(B),\quad \Pb_{W_{\mathrm{tri}}}(B) = \Pb_{W_{\mathrm{atri}}}(B).  \]
\end{enumerate}
Thus, if $(B,+)$ and $(B,\circ)$ are both abelian, then the sets in Definition \ref{defn:W} for all four families are the same, and the four probabilities are all equal.

\smallskip

In what follows, let $W$ be any of $W_{\mathrm{ann}},W_{\mathrm{soc}},W_{\mathrm{tri}}$ and $W_{\mathrm{atri}}$.

\begin{example}\label{ex:p2}Let $p$ be any prime, and consider the brace $B = (\mathbb{F}_p^2,+,\circ)$, where $+$ is the usual addition, and 
\[ \begin{pmatrix}a_1\\a_2\end{pmatrix}\circ \begin{pmatrix}b_1 \\ b_2\end{pmatrix} = \begin{pmatrix} {a_1} + b_1 + a_2b_2\\a_2+b_2\end{pmatrix}\]
(cf. \cite{bachiller}). We are in both situations (I) and (II). We easily see that
\[ F_{1,W}(B) = F_{2,W}(B) = \mathbb{F}_p\times \{0\},\]
which has prime index $p$ in $B$. It then follows from \eqref{eqn:prime} that
\[ \Pb_W(B) = \varphi(f_{2,W}(B),f_{1,W}(B)) = \varphi(p,p)= \frac{2p-1}{p^2},\]
and in particular $\Pb_W(B) = \frac{3}{4}$ when $p=2$.
\end{example}

\begin{example}\label{ex:p3}Let $p$ be any prime, and consider the brace $B = (\mathbb{F}_p^3,+,\circ)$, where $+$ is the usual addition, and 
\[\begin{pmatrix}
    a_1\\a_2\\a_3
\end{pmatrix}\circ \begin{pmatrix}b_1\\b_2\\b_3\end{pmatrix}
= \begin{pmatrix} 
a_1 + b_1 + \lvert \begin{smallmatrix}a_2&b_2\\a_3&b_3\end{smallmatrix} \rvert\\
a_2 + b_2 \\
a_3 + b_3\\
\end{pmatrix}
\]  
(cf. \cite{bachiller}). We are in situation (I). Since
\[ \forall a,b\in B: a \ast b = -(b\ast a),\]
although $(B,\circ)$ is not abelian, we still have that all of the sets in Definition~\ref{defn:W} are the same for $W=W_{\mathrm{ann}},W_{\mathrm{soc}}$. We easily see that
\[ F_{1,W}(B) = F_{2,W}(B) = \mathbb{F}_p\times \{0\}\times \{0\}.\]
Note that $\{0\}\times\mathbb{F}_p\times\mathbb{F}_p$ is a left transversal of $F_{1,W}(B)$ in $(B,\circ)$. Let $R$ be this set with the zero vector removed. It then follows from \eqref{eqn:Pb3} that
\[ \Pb_W(B) = \frac{1}{f_{1,W}(B)} + \frac{1}{f_{1,W}(B)} \sum_{a\in R} \frac{1}{c_W(a;B)} = \frac{1}{p^2} + \frac{1}{p^2}\sum_{a\in R}\frac{1}{c_W(a;B)}.\]
For each $a=(0,a_2,a_3)^T\in R$ and
$b=(b_1,b_2,b_3)^T\in B$, observe that
\begin{align*}
b\in C_W(a;B)
\iff a\ast b=0
\iff \left\lvert\begin{smallmatrix}a_2&b_2\\a_3&b_3 \end{smallmatrix}\right\rvert=0
\iff \left(\begin{smallmatrix}b_2\\b_3\end{smallmatrix}\right)\in
\left\langle\Big(\begin{smallmatrix}a_2\\a_3\end{smallmatrix}\Big)\right\rangle .
\end{align*}
Since $(a_2,a_3)\neq (0,0)$, we see that
\[ c_{W}(a;B)=\frac{|B|}{|C_{W}(a;B)|} = \frac{p^3}{p^2} = p.\]
Therefore, we deduce that
\[ \Pb_W(B) = \frac{1}{p^2} + \frac{1}{p^2}(p^2-1)\frac{1}{p} = \frac{p^2+p-1}{p^3},\] so in particular $\Pb_W(B) = \frac{5}{8}$ when $p=2$.
\end{example}

\begin{example}\label{ex:pq} 
Let $p,q$ be primes such that $p\equiv 1\pmod{q}$, and consider the brace $B = (\mathbb{F}_p\times\mathbb{F}_q,+,\circ)$, where $+$ is the usual addition, and
\[ \begin{pmatrix}a_1\\a_2\end{pmatrix}\circ \begin{pmatrix}b_1 \\ b_2\end{pmatrix} = \begin{pmatrix} {a_1} + k^{a_2}b_1 \\a_2+b_2\end{pmatrix}\]
(cf. \cite{pq}). In particular, we have
\[\begin{pmatrix}a_1\\a_2\end{pmatrix}\ast\begin{pmatrix}b_1 \\ b_2\end{pmatrix} = \begin{pmatrix} (k^{a_2}-1) b_1\\0
\end{pmatrix}.\]
Here $k\in \mathbb{F}_p^\times$ is a fixed element of order $q$. We are in situation (I).

\smallskip

For $W = W_{\operatorname{ann}}$, we have that 
\[ C_W\left( \left(\begin{smallmatrix}a_1\\a_2\end{smallmatrix}\right);B\right) = \begin{cases}
 \mathbb{F}_p\times \mathbb{F}_q&\mbox{when }a_1,a_2=0,\\
 \{0\}\times \mathbb{F}_q &\mbox{when $a_1=0,a_2\neq 0$},\\
   \mathbb{F}_p\times \{0\} &\mbox{when $a_1\neq 0,a_2=0$},\\
   \{0\}\times \{0\} &\mbox{when $a_1,a_2\neq 0$}.
\end{cases}\]
Therefore, we deduce from the definition that
\begin{align*}
\Pb_W(B)& = \frac{1}{|B|^2}\sum_{a\in B} |C_W(a;B)|\\
&= \frac{pq+(q-1)q+(p-1)p+(q-1)(p-1)}{p^2q^2},
\end{align*}
so in particular $\Pb_W(B)=\frac{4}{9},\frac{9}{25}$ when $(p,q)=(3,2),(5,2)$, respectively.

\smallskip

For $W =W_{\operatorname{soc}},W_{\operatorname{tri}},W_{\operatorname{atri}}$, we easily see that
\begin{align*} F_{1,W}(B)  = \mathbb{F}_p\times \{0\},\quad 
F_{2,W}(B)  =\{0\} \times \mathbb{F}_q,
\end{align*}
and the former has prime index $q$ in $B$.
It then follows from \eqref{eqn:prime} that
\[ \Pb_W(B) = \varphi(f_{2,W}(B),f_{1,W}(B)) = \varphi(p,q) = \frac{p+q-1}{pq},\]
so in particular $\Pb_W(B) = \frac{2}{3},\frac{3}{5}$ when $(p,q)=(3,2),(5,2)$, respectively. 
\end{example}

Example \ref{ex:pq} above also shows that the ratios
\begin{equation}\label{eqn:ratio1} \frac{\Pb_{W_{\operatorname{ann}}}(B)}{\Pb_{W_{\operatorname{soc}}}(B)},\quad \frac{\Pb_{W_{\operatorname{ann}}}(B)}{\Pb_{W_{\operatorname{tri}}}(B)},\quad \frac{\Pb_{W_{\operatorname{ann}}}(B)}{\Pb_{W_{\operatorname{atri}}}(B)}\end{equation}
can be very small, while Example \ref{ex:final} below shows that the ratios
\begin{equation}\label{eqn:ratio2} \frac{\Pb_{W_{\operatorname{soc}}}(B)}{\Pb_{W_{\operatorname{tri}}}(B)}
,\quad\frac{\Pb_{W_{\operatorname{soc}}}(B)}{\Pb_{W_{\operatorname{atri}}}(B)}\end{equation}
can also be very small. The ratios in \eqref{eqn:ratio1} and \eqref{eqn:ratio2}, respectively, have the same value in Examples \ref{ex:pq} and \ref{ex:final}. Both values are equal to
\[ \frac{pq + (q-1)q+(p-1)p+(q-1)(p-1)}{pq(p+q-1)}
,\]
which tends to $\frac{1}{q}$ as $p\rightarrow\infty$ when $q$ is fixed. Note that $p\rightarrow\infty$ makes sense, because $p\equiv 1\pmod{q}$ has infinitely many prime solutions $p$ by Dirichlet’s Theorem on Arithmetic Progressions.

\begin{example}\label{ex:final}
Let $p,q$ be primes such that $p\equiv 1\pmod{q}$, and consider the skew brace $B = (\mathbb{F}_p\times \mathbb{F}_q,+,\circ)$, where
\[
    \begin{pmatrix}a_1\\a_2\end{pmatrix} 
    +\begin{pmatrix}b_1\\b_2\end{pmatrix} =\begin{pmatrix} a_1 + k^{a_2}b_1\\
    a_2+b_2
    \end{pmatrix},\quad
    \begin{pmatrix}a_1\\a_2\end{pmatrix} \circ\begin{pmatrix}b_1\\b_2\end{pmatrix}= \begin{pmatrix}
        k^{b_2}a_1 + k^{a_2}b_1\\a_2+b_2
    \end{pmatrix}
\]
(cf. \cite{pq}). In particular, we have
\begin{align*}\begin{pmatrix}a_1\\a_2\end{pmatrix}\ast\begin{pmatrix}b_1 \\ b_2\end{pmatrix} &= \begin{pmatrix} k^{-a_2}(k^{b_2}-1) a_1\\0
\end{pmatrix},\\
\left[\begin{pmatrix}a_1\\a_2\end{pmatrix},\begin{pmatrix}b_1 \\ b_2\end{pmatrix}\right]_+ & = \begin{pmatrix}
    (1-k^{b_2})a_1 + (k^{a_2}-1)b_1\\0
\end{pmatrix}.\end{align*}
Here $k\in \mathbb{F}_p^\times$ is a fixed element of order $q$. We are in situation (II).

\smallskip

 For $W=W_{\operatorname{ann}},W_{\operatorname{soc}}$, observe that 
\[ C_W\left( \left(\begin{smallmatrix}a_1\\a_2\end{smallmatrix}\right);B\right) = \begin{cases}
 \mathbb{F}_p\times \mathbb{F}_q&\mbox{when }a_1,a_2=0,\\
 \{0\}\times \mathbb{F}_q &\mbox{when $a_1=0,a_2\neq 0$},\\
   \mathbb{F}_p\times \{0\} &\mbox{when $a_1\neq 0,a_2=0$},\\
   \{0\}\times \{0\} &\mbox{when $a_1,a_2\neq 0$}.
\end{cases}\]
Therefore, we deduce from the definition that
\begin{align*}
\Pb_W(B)& = \frac{1}{|B|^2}\sum_{a\in B} |C_W(a;B)|\\
&= \frac{pq+(q-1)q+(p-1)p+(q-1)(p-1)}{p^2q^2},
\end{align*}
which is equal to the first probability computed in Example \ref{ex:pq}.

\smallskip

For $W =W_{\operatorname{tri}},W_{\operatorname{atri}}$, we easily see that
\begin{align*}
     F_{1,W_{\mathrm{tri}}}(B) &= F_{2,W_{\mathrm{atri}}}(B) =  \{0\} \times \mathbb{F}_q,\\
     F_{2,W_{\mathrm{tri}}}(B) & = F_{1,W_{\mathrm{atri}}}(B) = \mathbb{F}_p\times\{0\}. 
\end{align*}
Since $F_{1,W}(B)$ has prime index in $B$ in both cases, and $\varphi$ is symmetric, it then follows from \eqref{eqn:prime} that
\[ \Pb_W(B) = \varphi(f_{2,W}(B),f_{1,W}(B)) = \varphi(p,q) = \frac{p+q-1}{pq},\]
which is equal to the second probability computed in Example \ref{ex:pq}.
\end{example}

\section{A probability for indecomposable components}\label{sec:indecomp}
\label{sec:indecomposable-components-probability}

Let $(Z,r)$ be a solution. A {\it decomposition of $Z$} is a partition 
\[ Z=X\sqcup Y\mbox{ such that }\begin{cases}
\begin{aligned}
r(X\times X) &= X\times X,&r(X\times Y) &= Y\times X,\\
r(Y\times Y) &= Y\times Y,&r(Y\times X) &= X\times Y.
\end{aligned}
\end{cases}\]
It is easy to check that $(X,r|_{X\times X})$ and $(Y,r|_{Y\times Y})$ are also solutions in this case. We say that $(Z,r)$ is {\it decomposable} if it admits such a decomposition, and {\it indecomposable} otherwise. Recall that the {\it permutation group of $Z$} is 
\[ 
\mathcal G(Z,r)=\langle \lambda_z,\rho_z\mid z\in Z\rangle,
 \]
defined as a subgroup of $\mathrm{Sym}(Z)$. The following proposition is well-known, and the proof is also straightforward. Thus $(Z,r)$ is indecomposable if and only if $\mathcal{G}(Z,r)$ acts transitively on $Z$.

\begin{prop}\label{prop:decomposition}Let $(Z,r)$ be a solution. A partition $Z=X\sqcup Y$ of $Z$ is a decomposition  if and only if both $X,\, Y$ are unions of orbits of $\mathcal{G}(Z,r)$.
\end{prop}

For example, let $Z$ be a non-empty set. Given any $\sigma\in \mathrm{Sym}(Z)$, the map
\[ r_\sigma : (x,y) \in Z\times Z \longmapsto (\sigma(y),\sigma^{-1}(x))\in Z\times Z\]
is a solution with $\lambda_x=\sigma,\, \rho_y=\sigma^{-1}$ for all $x,y\in Z$. Then $\mathcal{G}(Z,r_\sigma) = \langle\sigma\rangle$, and $(Z,r_\sigma)$ is indecomposable if and only if $\sigma$ acts transitively on $Z$.


\smallskip

Let us introduce a new probability.

\begin{definition}
Let $(Z,r)$ be a finite solution, and let $X,Y$ be non-empty
subsets of $Z$. In the notation of Definition \ref{defn:W}, we define
\[
        \Pb_{\operatorname{ic}}(X,Y,r)
        =
        \Pb_{W_{\operatorname{ic}}}(X,Y),
\]
where 
\[
        W_{\operatorname{ic}}=\{\,x\mbox{ and $y$ belong to the same orbit of $\mathcal{G}(Z,r)$}\,\}.
\]
From Proposition \ref{prop:decomposition}, it is obvious that
\begin{align*}
Z = X\sqcup Y \mbox{ is a decomposition} &\iff  \Pb_{\mathrm{ic}}(X,Y,r) =0,
\end{align*}
and
\begin{align}\label{eqn:XY}
Z = X\sqcup Y\mbox{ is not a decomposition} &\iff \Pb_{\operatorname{ic}}(X,Y,r) \geq \frac{1}{|X||Y|},
\end{align}
so this probability measures how close $Z = X\cup Y$ is from being a decomposition. For $X=Y=Z$, we write
\[         \Pb_{\operatorname{ic}}(Z,r)
        =\Pb_{\operatorname{ic}}(Z,Z,r).\]
Again from Proposition \ref{prop:decomposition}, it is obvious that
\[ Z \mbox{ is indecomposable} \iff  \Pb_{\mathrm{ic}}(Z,r) =1,\]
so this probability measures how far $(Z,r)$ is from being indecomposable.
\end{definition}

We first give an example to show that the bound $\frac{1}{|X||Y|}$ in \eqref{eqn:XY} is sharp even when we restrict $X,Y$ to partitions $Z=X\sqcup Y$ of $Z$, and when we fix the sizes $|X|,|Y|$ in advance.

\begin{example}
\label{ex:ic-relative-gap-sharp}
Let $n_1$ and $n_2$ be any natural numbers. Let $Z$ be any finite set of size $n = n_1 + n_2$, and write
\[ Z = \{x_1,\dots,x_{n_1-1},u,y_1,\dots,y_{n_2-1},v\}.\]
Take $\sigma = (u,v)\in \mathrm{Sym}(Z)$ and consider $(Z,r_\sigma)$. Then
\[ Z=X\sqcup Y,\quad\mbox{for } X = \{x_1,\dots,x_{n_1-1},u\},\, Y = \{y_{1},\dots,y_{n_2-1},v\},\]
is a partition of $Z$ such that $|X| = n_1,\, |Y| = n_2$, and
\[ W_{\operatorname{ic}}(X,Y,r_\sigma) = \{(u,v)\},\quad \Pb_{\operatorname{ic}}(X,Y,r_\sigma) = \frac{1}{|X||Y|}.\]
Thus, the bound in \eqref{eqn:XY} is sharp for any sizes $|X|,\, |Y|$.
\end{example}

Next, we give estimates for the probability when $X=Y=Z$. We will use the following general combinatorial fact about partitions.

\begin{lem}\label{lem:partition-square-bounds} Let $n\in \mathbb{N}$, and let $n = n_1+ \cdots + n_m$ be a partition of $n$ into positive integers. Write
$n=mq+s$, with $0\leq s<m$. Then
\begin{equation}\label{eqn:partition-square-bounds}
        (m-s)q^2+s(q+1)^2
        \leq
        \sum_{i=1}^m n_i^2
        \leq
        (m-1)+(n-m+1)^2.
\end{equation}
Without loss of generality, assume that $n_1\leq \cdots \leq n_m$. Then:
\begin{enumerate}[$(a)$]
\item The left equality holds if and only if 
\[ n_1=\cdots = n_{m-s} = q,\quad n_{m-s+1} =\cdots =n_m = q+1.\]
\item The right equality holds if and only if
\[ n_1 = \cdots = n_{m-1} = 1,\quad n_{m} = n-m+1.\]
\end{enumerate}
\end{lem}

\begin{proof} For any partition $n=n_1+\cdots +n_m$ of $n$, define
\[ S(n_1,\dots,n_m) = \sum_{i=1}^m n_i^2.\]
We want to find the minimum and maximum of this sum, subject to \begin{equation}\label{eqn:condition}
 n_1 + \cdots + n_m = mq + s.
 \end{equation}
They will correspond, respectively, to the lower and upper bounds.

\smallskip

First, let $\mathbf{n}_{\mbox{\tiny min}}=(n_1,\ldots,n_m)$ be a partition of $n$ attaining the minimum. If $n_{j} - n_i\geq 2$ for some $i\neq j$, then for the partition $\mathbf{n}$ obtained from $\mathbf{n}_{\mbox{\tiny min}}$ by replacing $n_i$ with $n_i+1$ and $n_j$ with $n_j-1$, we have
\begin{align*}
S(\mathbf{n}_{\mbox{\tiny min}}) - S(\mathbf{n})
&= n_i^2 + n_{j}^2 - (n_i+1)^2 - (n_{j}-1)^2\\
&= 2(n_{j} -n_i-1)>0.
\end{align*}
This contradicts the minimality of $S(\mathbf{n}_{\mbox{\tiny min}})$. Thus, at the minimum, one has $|n_{j} - n_i|=0,1$ for all $i,j$. By \eqref{eqn:condition}, this is equivalent to the condition in (a) when we assume $n_1\leq \cdots\leq n_m$, and the lower bound follows.

\smallskip

Next, let $\mathbf{n}_{\mbox{\tiny max}} = (n_1,\ldots,n_m)$ be a partition of $n$ attaining the maximum. If $n_i,n_j \geq 2$ for some $i\neq j$, say $n_j \geq n_i$, then for the partition $\mathbf{n}$ obtained from $\mathbf{n}_{\mbox{\tiny max}}$ by replacing $n_i$ by $n_{i}-1$ and $n_{j}$ by $n_{j}+1$, we have
\begin{align*}
 S(\mathbf{n}_{\mbox{\tiny max}}) - S(\mathbf{n})
&= n_i^2 + n_{j}^2 - (n_i-1)^2 - (n_{j}+1)^2\\
&= 2(n_i-n_{j} -1)<0.
 \end{align*}
This contradicts the maximality of $S(\mathbf{n}_{\mbox{\tiny max}})$. Thus, at the maximum, one has $n_i=1$ for all but one $i$. By \eqref{eqn:condition}, this is equivalent to the condition in~(b) when we assume $n_1\leq \cdots \leq n_m$, and the upper bound follows.
\end{proof}

As an immediate application, we deduce the following result.

\begin{theorem}
\label{thm:component-bounds}
Let $(Z,r)$ be a finite solution of size $n$, and let
\[
        Z=O_1\sqcup\cdots\sqcup O_m
\]
be its decomposition into indecomposable components, i.e. the orbits under the action of $\mathcal{G}(Z,r)$. Write $n=mq+s$ with $0\leq s<m$. Then 
\[
\frac{(m-s)q^2+s(q+1)^2}{n^2}
 \leq\Pb_{\operatorname{ic}}(Z,r)\leq\frac{(m-1)+(n-m+1)^2}{n^2}.
\]
Moreover, without loss of generality, assume that $| O_1| \leq \ldots \leq |O_m|$. Then:
\begin{enumerate}[$(a)$]
\item The left equality holds if and only if 
\[|O_1|=\ldots =|O_{m-s}|=q,\quad |O_{m-s+1}| = \ldots =|O_m|=q+1.\]
\item The right equality holds if and only if
\[ |O_1|=\ldots =|O_{m-1}| =1,\quad |O_m| = n-m+1.\]
\end{enumerate}
\end{theorem}
\begin{proof} It is clear from the definition that
\[ \Pb_{\operatorname{ic}}(Z,r) = \frac{1}{n^2}\sum_{i=1}^{m} |O_i|^2. \]
The claims now follow by taking $n_i = |O_i|$,\, $i=1,\dots,m$, in Lemma \ref{lem:partition-square-bounds}.
\end{proof}

\begin{cor}
\label{cor:number-components}
Let $(Z,r)$ be a finite solution of size $n$.
\begin{enumerate}[$(a)$]
\item If
$\Pb_{\operatorname{ic}}(Z,r)<1/k$, where $k$ is any natural number, then $(Z,r)$ has at least $k+1$
indecomposable components.
\item If $\Pb_{\operatorname{ic}}(Z,r)>1-\frac{2(n-1)}{n^2}$, then $(Z,r)$ is indecomposable.
\end{enumerate}
\end{cor}
\begin{proof}
We use the same notation in Theorem~\ref{thm:component-bounds} and its proof.

\smallskip

 For (a), note that by the Cauchy--Schwarz inequality, we have
\begin{align*}
\Pb_{\operatorname{ic}}(Z,r) & = \frac{1}{n^2}\sum_{i=1}^mn_i^2
\geq \frac{1}{n^2}\cdot \frac{1}{m}\cdot \left(\sum_{i=1}^m n_i\right)^2 =\frac{1}{m}.
\end{align*}
Thus, if $\Pb_{\mathrm{ic}}(Z,r)<\frac{1}{k}$, then $m>k$ has to hold, namely, $(Z,r)$ has at least $k+1$ indecomposable components.

\smallskip

For (b), from Theorem \ref{thm:component-bounds}, we have the upper bound
\[ \Pb_{\operatorname{ic}}(Z,r)\leq \frac{(m-1)+(n-m+1)^2}{n^2}.\]
This is a decreasing function in $m$ for $1\leq m\leq n$, so when $m\geq 2$, we have
\[ \Pb_{\operatorname{ic}}(Z,r)\leq \frac{1+(n-1)^2}{n^2} = 1 -\frac{2(n-1)}{n^2}.\]
Thus, if $\Pb_{\mathrm{ic}}(Z,r)>1-\frac{2(n-1)}{n^2}$, then $m=1$ has to hold, namely, $(Z,r)$ is indecomposable.
\end{proof}

Finally, we show that both the lower and upper bounds in Theorem \ref{thm:component-bounds} are sharp, even when $n$ and $m$ are fixed in advance.

\begin{example}
\label{ex:ic-component-bound-sharp}
Let $m\leq n$ be any natural numbers, and write $n=mq+s$ with $0\leq s<m$. Let $Z$ be any finite set of size $n$. For any $\sigma,\tau\in\operatorname{Sym}(Z)$ having cycle types
\[ (\underbrace{q,\dots,q}_{\mbox{\tiny $m-s$ }},\underbrace{q+1,\dots,q+1}_{\mbox{\tiny $s$}}),\quad (\underbrace{1,\dots,1}_{\mbox{\tiny $m-1$}},n-m+1),\]
respectively, it follows from Theorem \ref{thm:component-bounds}(a),(b) that
\begin{align*}
 \Pb_{\operatorname{ic}}(Z,r_\sigma)
        &=\frac{(m-s)q^2+s(q+1)^2}{n^2},\\
\Pb_{\operatorname{ic}}(Z,r_{\tau})&=\frac{(m-1)+(n-m+1)^2}{n^2}.
\end{align*}
Thus, both the lower and upper bounds are sharp.
\end{example}

Example \ref{ex:ic-component-bound-sharp} also implies that $\Pb_{\mathrm{ic}}$ does not admit a uniform gap below~$1$, in the sense that there is no constant $\varepsilon>0$ such that either
\begin{equation}\label{eqn:ic iff} 
\Pb_{\operatorname{ic}}(Z,r)\leq 1-\varepsilon\quad\mbox{or}\quad \Pb_{\operatorname{ic}}(Z,r) =1
\end{equation}
for all finite solutions $(Z,r)$. Indeed, in Example \ref{ex:ic-component-bound-sharp}, take $m = 2$, so then
\[ \Pb_{\mathrm{ic}}(Z,r_{\tau}) = 1 - \frac{2(n-1)}{n^2}.\]
Letting $n\rightarrow\infty$ shows that \eqref{eqn:ic iff} does not hold for any constant $\epsilon >0$. It is worth emphasising that this non-discrete phenomenon below the value 1 is in stark contrast to Theorems A, B, and C.

\section{A probability for almost Yang--Baxter maps}
\label{sec:almost-yang-baxter-maps}

Instead of a solution, we start with an arbitrary bijective map
\[ r : (x,y) \in Z\times Z \longmapsto (\lambda_x(y),\rho_y(x))\in Z\times Z\]
that is {\it non-degenerate}, that is $\lambda_x,\, \rho_y$ are bijective for all $x,y\in Z$. We will call them {\it almost Yang--Baxter maps}, and we want to analyse when they are solutions from a probabilistic point of view. Unlike the probabilities in the previous sections, the condition for which $(Z,r)$ is a solution is $3$-variable.

\begin{definition}
Let $Z$ be any finite set, and let $r:Z\times Z\to Z\times Z$ be an almost Yang--Baxter map. Define
\[
        \Pb_{\operatorname{YB}}(Z,r)=\frac{1}{|Z|^3}|\mathrm{YB}(Z,r)|,        \]
where
\[
       \mathrm{YB}(Z,r)=
        \{(x,y,z)\in Z^3\mid
        r_{12}r_{23}r_{12}(x,y,z)=r_{23}r_{12}r_{23}(x,y,z)\}.
\]
We clearly have
\begin{align*}
(Z,r)\mbox{ is a solution}\iff \Pb_{\operatorname{YB}}(Z,r) = 1 ,
\end{align*}
so this probability measures how far $(Z,r)$ is from being a solution.
\end{definition}


We show that this probability, like $\Pb_{\mathrm{ic}}$, admits no uniform gap below~$1$, in the sense that
there is no constant $\varepsilon>0$ such that either
\begin{equation}\label{eqn:YB iff}
\Pb_{\mathrm{YB}}(Z,r)\leq 1-\epsilon\quad\mbox{or}\quad\Pb_{\mathrm{YB}}(Z,r) =1
\end{equation}
for all finite sets $Z$ and almost Yang--Baxter maps $r: Z\times Z\rightarrow Z\times Z$. In fact, we will show that  $\mathrm{YB}(Z,r)$ can be made to miss exactly $4$ elements of~$Z^3$, while $|Z|$ is allowed to be arbitrarily large.

\smallskip

To that end, let $Z$ be any non-empty set, and let $\pi: Z\rightarrow\mathrm{Sym}(Z)$ be any map. Clearly, we have an almost Yang--Baxter map
\[ r_\pi : (x,y) \in Z\times Z \longmapsto (\pi_x(y),x)\in Z\times Z\]
with $\lambda_x = \pi_x,\, \rho_y=\mathrm{id}_Z$ for all $x,y\in Z$, whose inverse is given by
\[ r_\pi^{-1} : (u,v)\in Z\times Z \longmapsto (v,\pi_v^{-1}(u))\in Z\times Z.\]
We first give a criterion for a triplet $(x,y,z)$ to belong to $\mathrm{YB}(Z,r_\pi)$.

\begin{lem}
\label{lem:left-permutation-yb-formula}
Let $Z$ be any non-empty set, and let $\pi: Z\rightarrow\mathrm{Sym}(Z)$ be any map. For any $x,y,z\in Z$, we have
\[
(x,y,z)\in \mathrm{YB}(Z,r_\pi)\iff \pi_{\pi_x(y)}\pi_x(z)=\pi_x\pi_y(z).
\]
In particular, this implies that
\[ (Z,r_\pi)\mbox{ is a solution}\iff \forall x,y\in Z:\pi_{\pi_x(y)}\pi_x=\pi_x\pi_y.\]
\end{lem}

\begin{proof}
Put $r = r_{\pi}$. A direct computation gives
\[
        r_{12}r_{23}r_{12}(x,y,z)
        =
        (\pi_{\pi_x(y)}\pi_x(z),\pi_x(y),x),
\]
while
\[
        r_{23}r_{12}r_{23}(x,y,z)
        =
        (\pi_x\pi_y(z),\pi_x(y),x).
\]
The claims are now clear.
\end{proof}


\begin{theorem}
\label{thm:no-uniform-yb-gap}
Let $n\geq 2$ be any natural number, and let  $Z$ be any finite set of size $n$. For any distinct $u,v\in Z$, there exists a map $\pi : Z\rightarrow \mathrm{Sym}(Z)$ such that
\[ \mathrm{YB}(Z,r_\pi) =  Z^3\setminus \{ (u,u,u),\, (u,u,v),\, (u,v,u),\, (u,v,v) \},\]
so in particular we have
\[
\Pb_{\operatorname{YB}}(Z,r_\pi)=1-\frac{4}{n^3}.
\]
Letting $n\rightarrow\infty$ shows that \eqref{eqn:YB iff} does not hold for any constant $\epsilon >0$.
\end{theorem}
\begin{proof} Define $\pi : Z\rightarrow\mathrm{Sym}(Z)$ by setting
\[\pi_x = \begin{cases}
(u,v)&\mbox{if }x = u,\\
\mathrm{id}_Z & \mbox{otherwise}.
\end{cases}\]
By Lemma~\ref{lem:left-permutation-yb-formula}, for any $x,y,z\in Z$, we have
\[ (x,y,z) \in \mathrm{YB}(Z,r_\pi) \iff 
\pi_{\pi_x(y)}\pi_x(z)=\pi_x\pi_y(z).\]
Consider this equality. 
\begin{enumerate}[$\bullet$]
\item If $x\neq u$, then both sides are equal to $\pi_y(z)$. 
\item If $x=u$ and $y\neq u,v$, then both sides are equal to $\pi_x(z)$.
\item If $x=u$ and $y=u$, then the left and right hand sides are
\begin{align*}
\pi_{\pi_u(u)}\pi_u(z) = \pi_v\pi_u(z) = (u,v)(z),\quad
\pi_u\pi_u(z) = z,
\end{align*}
respectively. They are equal exactly when $z \neq u,v$.
\item If $x=u$ and $y=v$, then the left and right hand sides are
\begin{align*}
\pi_{\pi_u(v)}\pi_u(z) = \pi_u\pi_u(z) = z,\quad \pi_u\pi_v(z) = (u,v)(z),
\end{align*}
respectively. Again, they are equal exactly when $z \neq u,v$.
\end{enumerate}
Thus $\operatorname{YB}(Z,r_\pi)$ is as claimed, and this proves the theorem.
\end{proof}

\section*{Acknowledgements}

This work is supported by JSPS KAKENHI 24K16891. Ferrara and Trombetti are members of the non-profit association ‘‘AGTA --- Advances in Group Theory and Applications’’ (www.advgrouptheory.com) and are supported by \hbox{GNSAGA} (INdAM).

\end{document}